\documentclass[a4paper]{amsart}
\usepackage[latin1]{inputenc}
\usepackage{amsmath}
\usepackage{amsthm}
\usepackage{amsfonts}
\usepackage{amssymb}
\usepackage[T1]{fontenc}
\usepackage{float}
\usepackage[all]{xy}
\usepackage{tikz}
\usepackage{tikz-cd}
\usetikzlibrary{arrows}
\usepackage{hyperref}
\usepackage{calrsfs}
\usepackage{bbm}
\usepackage{color}
\usepackage{colortbl}
\usepackage[hmargin=3cm,vmargin=4.4cm]{geometry}
\usepackage{tabularx}
\usepackage{bbm}
\usepackage{lscape}
\usepackage{comment}

\newtheorem{theorem}{Theorem}[section]
\newtheorem{prop}[theorem]{Proposition}

\newtheorem{corol}[theorem]{Corollary}

\theoremstyle{definition}
\newtheorem{defi}[theorem]{Definition}

\newtheorem{rmk}[theorem]{Remark}
\newtheorem{exmp}[theorem]{Example}

\def\dim{{\rm{dim}\,}}

\def\ens#1{\left\{ #1 \right\}}

\def\AA{{\mathcal{A}}}
\def\CC{{\mathcal{C}}}
\def\A{{\mathbb{A}}}

\def\Q{{\mathbb{Q}}}

\def\ens#1{\left\{ #1 \right\}}
\def\Ext{{\rm{Ext}}}

\def\Hom{{\rm{Hom}}}

\def\1{\mathbbm{1}}

\title{A new cluster character with coefficients for cluster category}
\author{Fernando Borges and Tanise Carnieri Pierin}
\address{Universidade Federal do Paran\'a, Curitiba, Brasil
}
\email{fernando.borges@ufpr.br}
\email{tanise@ufpr.br}
\date{\today}

\begin{document}

\begin{abstract}	
	
	We introduce a new cluster character with coefficients for a cluster category $\CC$ and rather than using a Frobenius $2$-Calabi-Yau realization to incorporate coefficients into the representation-theoretic model for a cluster algebra, as done by Fu and Keller, we exploit intrinsic properties of $\CC$. For this purpose, we define an ice quiver associated to each cluster tilting object in $\CC$. In Dynkin case $\A_n$, we also prove that the mutation class of the ice quiver associated to the cluster tilting object given by the direct sum of all projective objects is in bijection with set of ice quivers of cluster tilting objects in $\CC$.
\end{abstract}

\maketitle

\setcounter{tocdepth}{1}
\tableofcontents

\section*{Introduction}

In \cite{CC}, Caldero and Chapoton defined a map $X_?$ from cluster category $\CC_Q$ of a Dynkin quiver $Q$, with $n$ vertices, to $\Q(x_1,\cdots,x_n)$, which permits to obtain explicit formulas to compute cluster variables of the cluster algebra $\mathcal{A}_Q$ without coefficients. Caldero and Keller (\cite{CK06,CK08})  proved an essential property of Caldero-Chapoton map which is the equation \[X_NX_M=X_B+X_{B'}\] where $N$, $M$ are given indecomposable objects in the cluster category $\CC_Q$ such that $\Ext^1_{\CC_Q}(N,M)$ is one dimensional and $B$ and $B'$ are the middle terms of non-split triangles with outer terms $N$ and $M$. This multiplication formula gave rise to the concept of cluster character, as done in \cite{Palu}. In order to introduce the coefficients in the `categorification' of a class of cluster algebras with coefficients, Fu and Keller \cite{FK} replaced cluster categories, which are triangulated $2$-Calabi-Yau, by Frobenius $2$-Calabi-Yau, using the above-mentioned cluster characters. In \cite{Dupont}, Dupont gives a more elementary approach to cluster characters with coefficients.

In this paper, we generalize this latter concept, in order to be able to get a `categorification' of a larger class of cluster algebras with coefficients.    
An analogous result to the multiplication formula for cluster characters, in the sense of Palu, can be obtained and it is one of the fundamental tools to prove our main theorem, which will be stated later; as a consequence of this formula, for each rigid indecomposable object $V$ in $\CC_{\A_n}$, the cluster character with coefficients $X_V$ defined here (Definition \ref{clustercharacter}) can be written in terms of the determinant of a tridiagonal matrix whose principal diagonal consists of the characters of simple modules or the shift of the simple projective module,  and the others are coefficients. More precisely, by indexing the indecomposable objects in $\CC_{\A_n}$ by $(i,j)=\tau^{-i}P_j$, where $P_j$ is the projective object associated to the vertex $j$,  $1\leq j\leq n$ and  $	0\leq i \leq n-j+1$, it is possible to write

\[X_{(i,j)}=\left|\begin{smallmatrix}
                      X_{(i,1)}&y_{i+2}&0&\cdots &0\\
                      z_{i+1}&X_{(i+1,1)}&y_{i+3}&\cdots &0\\
                      0      &z_{i+2}&\ddots&& \vdots\\
                      \vdots &       &       && y_{i+j} \\
                      0      &\cdots &       0&z_{i+j-1}&X_{(i+j-1,1)}
                      \end{smallmatrix}\right| \mbox{ and }
                      z_1\cdots z_{j-1}x_j=\left|\begin{smallmatrix}
                      x_1&y_1&0&\cdots &0\\
                      1&X_{(0,1)}&y_2&\cdots &0\\
                      0      &z_1&\ddots&& \vdots\\
                      \vdots &       &       && y_{j-1} \\
                      0      &\cdots &       0&z_{j-2}&X_{(j-2,1)}                    
\end{smallmatrix}\right|;\] 
$y$'s and $z$'s are coefficients. Observe that these formulas are similar to those ones obtained in \cite{Dupont} for quantized Chebyshev polynomials and in \cite{BRS} for $c$-continuant polynomials. 
 % Em particular, as variáveis de conglomerado não iniciais de uma álgebra de conglomerado com coeficientes, associada a uma aljava emoldurada do tipo $\A_n$, são polinômios nos caracteres dos módulos simples e nos coeficientes $y_1,\cdots, y_{2n}$. 

In this work we are particularly interested in cluster algebras with coefficients associated to  ice quivers that we call biframed quivers, defined as follows. Given a quiver $Q$ such that its set of vertices is $Q_0=\{1,2,\cdots, n\}$,  a biframed quiver associated to $Q$ is the ice quiver $\widehat{Q}=(\widehat{Q}_0,\widehat{Q}_1)$  of type $(n,3n)$, where $\widehat{Q}_0=Q_0\cup \{n+1,\cdots,3n\}$ and $\widehat{Q}_1$ is the union of $Q_1$ and the set consisting of arrows $2n+i\rightarrow i$ and $i\rightarrow n+i$, for $i=1,\cdots,n$. Notice that cluster algebras with principal coefficients are obtained as a particular case of the above considered by taking half of the coefficients equal to one. 
For a biframed quiver associated to the Dynkin quiver $\A_n$, we prove the following theorem. 

\medskip

\noindent {\bf Theorem A}. The cluster variables of a cluster algebra with  coefficients associated to the biframed quiver $\widehat{\mathbb{A}}_n$ are precisely the characters $X_V$ when $V$ runs over the indecomposable rigid objects in $\mathcal{C}_{\A_n}$.

\medskip
 
 In particular, by combining this result and the previous formulas for cluster characters $X_V$, the cluster variables of the cluster algebra $\mathcal{A}(\textbf{x},\widehat{\A}_n)$ are polynomials over $\Q[\textbf{y},\textbf{z}^{\pm 1}]$ in the characters of the simple modules and the character of the shift of the simple projective, that is, \[\mathcal{A}(\textbf{x},\widehat{\A}_n)=\Q[\textbf{y},\textbf{z}^{\pm 1}][x_1,X_{S_1},\cdots,X_{S_n}].\]  
 
\noindent We point out that for obtaining all cluster variables of the cluster algebra associated to a biframed quiver  $\widehat{\A}_n$ we are using the cluster category $\CC$ associated to the principal part $\A_n$ of $\widehat{\A}_n$. In order to prove Theorem A we propose a definition of an ice quiver $Q_T$ associated to each cluster tilting object $T$ in $\CC$; this definition led us to a bijection between cluster tilting objects of $\CC$ and the mutation class of $\widehat{\mathbb{A}}_n$, as stated in the following theorem.

\medskip 
 
\noindent {\bf Theorem B}. Let $\widehat{\mathbb{A}}_n$ be the biframed quiver of type $\mathbb{A}_n$ and let $\mathcal{C}=\mathcal{C}_{\mathbb{A}_n}$ be the cluster category associated to the principal part of $\widehat{\mathbb{A}}_n$. There exists a bijection between  the cluster tilting objects of $\mathcal{C}$ and the mutation class of $\widehat{\mathbb{A}}_n$ given by $T\mapsto Q_T$.  

\medskip

This correspondence given by Theorem B is compatible with mutations of quiver and cluster tilting object, in the sense that, for a cluster tilting object $T$ in $\CC_{\A_n}$, the following equality holds $\mu_r(Q_T)=Q_{\mu_r(T)}.$ Moreover, by composing the inverse of map $T\mapsto Q_T$ with  cluster character with coefficients map $X_?$ we obtain that any seed is determined by its ice quiver; and also, as usual, any seed is determined by its cluster.  

\medskip

For the others Dynkin types, we obtained the multiplication formula (mentioned before) for almost split triangles, which enable us to prove the following theorem.

\medskip
\noindent {\bf Theorem C}. Let $Q$ be a quiver of Dynkin type.  The characters $X_V$ when $V$ runs over the indecomposable rigid objects in $\mathcal{C}_Q$ are cluster variables of a cluster algebra with  coefficients associated to the biframed quiver $\widehat{Q}$.

\newpage

\section{Preliminaries}
\subsection{Cluster algebra}
\subsubsection{Mutation and seeds}\label{mutationQ}

Let us recall that a \textit{quiver} $Q$ is a quadruple $(Q_0,Q_1,s,t)$,  where $Q_0$ and $Q_1$ are sets whose elements are called \textit{vertices} and \textit{arrows}, respectively, and $s,t$ are maps from $Q_1$ to $Q_0$ which associate to each arrow $\alpha$ its \textit{source} $s(\alpha)$ and its \textit{target} $t(\alpha)$. A finite quiver without loops and $2$-cycles is called a \textit{cluster quiver}. Given such a quiver $Q$,  the \textit{mutation}  of $Q$ at vertex $r$ is a new quiver $\mu_r(Q)$ obtained by following the steps:

\begin{enumerate}
    \item[$(i)$]  for each path $i\rightarrow r\rightarrow j$ create a new arrow $i\rightarrow j$,
    \item[$(ii)$]  reverse all arrows which have $r$ as source or target,
    \item[$(iii)$]  delete all the $2$-cycles created.
\end{enumerate}

Given $1\leq n\leq m$ integers, an \textit{ice quiver} of type $(n,m)$ is a  pair $\widehat{Q}=(Q,F)$ where $Q$ is a cluster quiver with $|Q_0|=m$ and $F$ is a subset of $Q_0$ such that $|F|=m-n$, whose elements  are called \textit{frozen vertices}. The mutation of an ice quiver $\widehat{Q}$ is defined only at non frozen vertices and coincides with quiver mutation in the sense mentioned before. For simplicity, we will always assume that $Q_0=\{1,\cdots,m\}$ and $F=\{n+1,\cdots,m\}.$

A \textit{seed} is a pair $(Q,u)$ where $Q$ is an ice quiver of type $(n,m)$ and $u=\{u_1,\cdots, u_m\}$ is a set consisting of $m$ algebraically independent variables over $\Q$. For a seed $(Q,u)$, the \textit{seed mutation} at $r$ is the new seed $\mu_r(Q,u)=(Q',u')$ with $Q'$  obtained from $Q$ by quiver mutation at vertex $r$  and $u'=u\setminus\{u_r\}\cup\{u'_r\}$ where $u'_r$ is defined by the \textit{exchange relation}

\begin{equation}\label{exrel}
  u_ru_r'=\prod_{v\rightarrow r}u_v +\prod_{r\rightarrow v}u_v.
\end{equation}

\subsubsection{Cluster algebra with coefficients} Let $(Q,u)$ be a seed, which we will call the initial seed.
\begin{itemize}
  \item  A cluster is any set $u'$ appearing in a seed $(Q', u')$ obtained from the initial
seed by a finite sequence of mutation.
\item A cluster variable is any element of a cluster associated to a non frozen vertex and the coefficients are the variables associated to the frozen ones.
\item  The cluster algebra associated to $(Q,u)$, denoted by $\AA(Q,u),$ is the $\Q[x_{n+1},\cdots,x_m]$-subalgebra  of the field of
rational functions $\Q(x_1 , x_2 , \dots, x_m )$ generated by the set of all cluster variables. 
\end{itemize}

\subsubsection{Principal coefficients} Let $Q$ be an ice quiver of type $(n,2n)$ and $B$ be the skew-symmetric matrix correspondent. We say that the cluster algebra $\mathcal{A}(\textbf{x}, B)$ has \textit{principal coefficients} if \[B=\left(\begin{matrix}
  B' & -\mathrm{id}\\ \mathrm{id}& 0 
\end{matrix}\right)\] where $\mathrm{id}$ denotes the $n\times n$ identity matrix.

\subsubsection{Biframed quiver}\label{Biframed} Let $Q$ be a quiver such that its set of vertices is $Q_0=\{1,2,\cdots, n\}$. The biframed quiver associated to $Q$ is the ice quiver $\widehat{Q}=(\widehat{Q}_0,\widehat{Q}_1)$  of type $(n,3n)$, where $\widehat{Q}_0=Q_0\cup \{n+1,\cdots,3n\}$ and $\widehat{Q}_1$ is the union of $Q_1$ and the set consisting of arrows $2n+i\rightarrow i$ and $i\rightarrow n+i$, for $i=1,\cdots,n$.  For the cluster algebra associated to $\widehat{Q}$, we refer the coefficients associated to vertices $n+1,\dots, 2n$ as $y's$ and the remaining ones as $z's$. Observe that  any cluster algebra with principal coefficients is a particular case of a cluster algebra associated to a biframed quiver by specializing half of its coefficients to one.

\subsection{Cluster Category} For a given  algebraically closed field $k$ and an acyclic quiver $Q$, let $\mathcal{D}=\mathcal{D}^b(\mathrm{mod}\   kQ)$ be the bounded derived category of $\mathrm{mod}\ kQ$, $\tau_{\mathcal{D}}$ its Auslander-Reiten translation and $[1]_{\mathcal{D}}$ the shift functor. The \textit{cluster category}  $\CC_Q$ is the orbit category of the functor $\tau^{-1}_{\mathcal{D}}[1]$. It is well known that $\CC_Q$ is a triangulated category \cite{keller}; the shift and the Auslander-Reiten translation in $\CC_Q$ are induced by shift and Auslander-Reiten translation in $\mathcal{D}$. 

\subsubsection{Cluster Tilting Objects and Mutation}\label{mutationT}  
In the above conditions, assuming that  $Q$ has  $n$ vertices we say that an object $T$ in  $\CC_Q$ is  cluster tilting if

\begin{itemize}
	\item[ $(i)$] $\Ext_{\CC_Q}^1(T,T)=0$;
	\item[ $(ii)$] $T = T_1\oplus\cdots\oplus T_n$, where the $T_r$'s are nonisomorphic indecomposable objects.
\end{itemize}
An object $T$ in $\CC_Q$ which satisfies only condition $(i)$ is called {\it rigid}.

\medskip
Let $T = T_1\oplus\cdots\oplus T_n$ be a cluster tilting object in $\CC_Q$. For each $r = 1,\cdots, n$, the mutation of $T$ at $r$ is defined by  $\mu_r(T)=T/T_r\oplus T'_r$, where $T'_r\neq T_r$ is the unique indecomposable object 
in $\CC_Q$ such that  $T/T_r\oplus T'_r$ is a cluster tilting object in $\CC_Q$. By taking the minimal right $\rm{add}(T/T_r)$-approximations of $T'_r$ and $T_r$ we obtain the following exchange triangles \[T_r\rightarrow B\rightarrow T'_r\rightarrow T_r[1]\  \mbox{and}\ T'_r\rightarrow B'\rightarrow T_r\rightarrow T'_r[1].\] In the sequel, unless stated otherwise, we will always use the notation above for the exchange pair $(T_r,T'_r).$ 

According  to \cite{BMRRT} it is possible to associate to $\CC_Q$ a connected graph whose vertices
are the cluster tilting objects, and where there is an edge between two vertices if
the corresponding cluster tilting objects have all but one indecomposable summands in
common.

\section{Cluster character with coefficients}
\subsection{Definition}
\label{clustercharacter} Let $Q$ be an acyclic finite quiver.  
The \textit{cluster character with coefficients} \[X_?: obj\mathcal{C}_Q\rightarrow\Q (x_1,\dots,x_n,y_1,\dots,y_{n},z_1,\dots,z_n)\] is defined as a map which satisfies:

\begin{itemize}

\item For all $M,N\in obj\mathcal{C}_Q$, we have $X_{M\oplus N}=X_MX_N$.
\item  $X_{P_i[1]}=x_i,$ where $P_i[1]$ is the shift of an indecomposable projective  module.
\item If $V$ is an indecomposable   $kQ$-module with dimension vector  $d=\displaystyle\sum_{i\in Q_0}d_i\alpha_i=(d_1,\dots,d_n)$, then its  cluster character with coefficients is given by
\end{itemize}
\[
 X_V= \sum_{0\leq e\leq d}\chi(Gr_e(V))\prod_{r=1}^nx_r^{-\langle e,\alpha_r\rangle -\langle \alpha_r,d-e\rangle}y_r^{e_r}z_{r}^{d_r-e_r}\] 
where  $\chi$ denotes the Euler characteristic and $\langle\ ,\ \rangle$ the Euler form.

For a set of variables $\ens{u_1,\dots,u_n}$ and a dimension vector $d=(d_1,\dots,d_n)$ we sometimes use the notation $u^d$ to represent the product $u_1\cdot u_2\cdots u_n$; in this way, the cluster character becomes 

\[
 X_V= \sum_{0\leq e\leq d} y^ez^{d-e}\chi(Gr_e(V))\prod_{r=1}^nx_r^{-\langle e,\alpha_r\rangle -\langle \alpha_r,d-e\rangle}.\]

\begin{rmk}
\end{rmk}
\begin{itemize}
 \item If $y_r=1,z_r=1$,  $1\leq r\leq n$, then $X_?$ is the Caldero-Chapoton map.
 \item If $z_r=1$, $1\leq r\leq n$, then $X_?$  is the cluster character with coefficients defined by Dupont (2010). In particular, the cluster variables of a cluster algebra with principal coefficients are precisely the characters $X_V$ when $V$ runs over the indecomposable rigid objects in $\mathcal{C}_Q$.
\end{itemize}
  
Taking this remark into account, it becomes natural to ask if there exists a cluster algebra which relates to the above defined cluster character, as happens in the classical case. In the next section we give a partial answer for Dynkin case.
%it is a natural question if there is a cluster algebra which relates to the above defined cluster character. In the next section we give a partial answer for Dynkin case. 

\section{Dynkin case $\mathbb{A}_n$}

The aim of this section is to  establish the connection between cluster character and mutation of a seed associated to the biframed quiver $Q$ of type $\A_n$; using this, it will be possible to prove that the set of  cluster variables of the cluster algebra $\mathcal{A}_Q$ coincides with the set of characters $X_V$ when $V$ runs over the indecomposable rigid objects in the cluster category. 

\subsection{A multiplication formula}\label{sec:multiAn}
For Dynkin case $\A_n$, the exchange triangles mentioned in Subsection \ref{mutationT}  have the  configuration presented in Figure \ref{fig:cmrg}, where $B$ is either $U_1\oplus U_2$ or $U_3\oplus U_4$ depending on whether $T_r= N$ or $T_r = M$. 
 
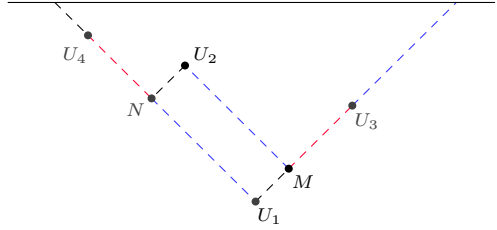
\begin{figure}[htb]
				\begin{center}
\definecolor{ttttff}{rgb}{0.2,0.2,1}
\definecolor{ffqqtt}{rgb}{1,0,0.2}
\definecolor{xdxdff}{rgb}{0.49,0.49,1}
\definecolor{uququq}{rgb}{0.25,0.25,0.25}
\begin{tikzpicture}[line cap=round,line join=round,>=triangle 45,x=0.40cm,y=0.40cm]
%\clip(-8.16,-13.05) rectangle (17.08,8.13);
\draw [dashed](-1.96,3.54)-- (-0.87,2.45);
\draw [dashed,color=ffqqtt] (-0.87,2.45)-- (1.22,0.36);
\draw [dashed,color=ttttff] (1.22,0.36)-- (4.64,-3.06);
\draw [dashed](4.64,-3.06)-- (5.73,-1.97);
\draw [dashed,color=ttttff] (5.73,-1.97)-- (2.32,1.45);
\draw [dashed](2.32,1.45)-- (1.22,0.36);
\draw [dashed,color=ffqqtt] (5.73,-1.97)-- (7.82,0.12);
\draw [dashed,color=ttttff] (7.82,0.12)-- (11.24,3.54);
\draw (-3.5,3.54)-- (12.5,3.54);
\begin{scriptsize}
\fill [color=uququq] (4.64,-3.06) circle (1.5pt);
\draw[color=black] (5.1,-3.54) node {$U_1$};
\fill [color=black] (5.73,-1.97) circle (1.5pt);
\draw[color=black] (6.2,-2.4) node {$M$};
\fill [color=black] (2.32,1.45) circle (1.5pt);
\draw[color=black] (3.0,1.79) node {$U_2$};
\fill [color=uququq] (1.22,0.36) circle (1.5pt);
\draw[color=uququq] (0.7,0) node {$N$};
\fill [color=uququq] (-0.87,2.45) circle (1.5pt);
\draw[color=uququq] (-1.3,1.8) node {$U_4$};
\fill [color=uququq] (7.82,0.12) circle (1.5pt);
\draw[color=uququq] (8.3,-0.4) node {$U_3$};
\end{scriptsize}
\end{tikzpicture}
\end{center}
\caption{ Exchange triangles for $\A_n$}\label{fig:cmrg}
\end{figure}
\noindent 
In the next theorem it is given a multiplication formula for the considered case. For its proof we consider the same notation as in Figure \ref{fig:cmrg}.

\begin{theorem}\label{teo:multf}
  Let $H^0=\Hom(\oplus_i P_i,\_):\mathcal{C}_{\A_n}\rightarrow\mathrm{mod}k\A_n$ and let $M$ and $N$ be indecomposable objects in $\mathcal{C}_{\A_n}$ such that $\Ext^1_{\mathcal{C}_{\A_n}}(M,N)$ is one-dimensional. 
\begin{itemize}
 \item[$(a)$]
   If $N$ and $M$ are 	 modules such that $\Hom_{k\A_n}(M,N)=0$, then there exist modules $U_3,U_4$ such that $H^0(U')=U_3\oplus U_4$ and  \[X_NX_M=X_{U}+X_{U'}y^{\underline{\dim}M-\underline{\dim}U_3}z^{\underline{\dim}N-\underline{\dim}U_4};\]

\item[$(b)$] If $M=P_i[1]$  and $N$ is an 	indecomposable module, then  \[X_NX_M=X_{U}y^{\underline{\dim}N-\underline{\dim}H^0U}
 +X_{U'}z^{\underline{\dim}N -\underline{\dim}H^0U'},
\]

\end{itemize}  where $U$ and $U'$ are the unique objects such that there exist non split triangles  \[N\rightarrow U\rightarrow M\rightarrow N[1]\  \mbox{and}\ M\rightarrow U'\rightarrow N\rightarrow M[1].\]
\end{theorem}
\begin{proof} 

We first point out that if $M$ is an indecomposable module, then $\chi (Gr_e(M))$ is either $0$ or $1$, depending on whether $M$ has a submodule with dimension vector $e$  (\cite{CC}, example 3.2). By Definition \ref{clustercharacter}, setting $\underline{n} = \dim N$ and $\underline{m} = \dim M$,

\begin{equation}\label{eq:NM}
 X_{N\oplus M}= \sum_{\substack{ 0\leq e\leq \underline{n}\\  0\leq f\leq \underline{m}} }y^{e+f}z^{\underline{n}+\underline{m} -(e+f)}\prod_{r=1}^nx_r^{-\langle e +f,\alpha_r\rangle -\langle \alpha_r,\underline{n}+\underline{m} -(e+f)\rangle}.
 \end{equation}
Since there is an exact sequence $0\rightarrow N\rightarrow U_1\oplus U_2\rightarrow M\rightarrow 0$, it is possible to prove that $u_1+u_2=\underline{n}+\underline{m}$, where $u_1$ and $u_2$ are the dimension vectors of $U_1$ and $U_2$, respectively. In this case
\begin{equation}\label{eq:U}
 X_{U_1\oplus U_2}=  \sum_{\substack{ 0\leq e\leq u_1\\  0\leq f\leq u_2} }y^{e+f}z^{\underline{n}+\underline{m} -(e+f)}\prod_{r=1}^nx_r^{-\langle e +f,\alpha_r\rangle -\langle \alpha_r,\underline{n}+\underline{m} -(e+f)\rangle}.
 \end{equation}
On the other hand,
\[
 X_{U_3\oplus U_4}y^{\underline{m}-u_3}z^{\underline{n}-u_4}= \sum_{\substack{ 0\leq e\leq u_4\\  0\leq g\leq u_3} }y^{e+g+\underline{m}-u_3}z^{\underline{n}+u_3 -(e+g)}\prod_{r=1}^nx_r^{-\langle e +g,\alpha_r\rangle -\langle \alpha_r,u_3+u_4 -(e+g)\rangle}.\]
For each submodule $G$ of $U_3$, with dimension vector denoted by $g$, there exists an exact sequence $0\rightarrow F\rightarrow M\oplus G\rightarrow U_3\rightarrow 0$, where $U_2\subsetneq F\subset M$. As a consequence, letting $u_3$ and $f$ the dimension vectors of $U_3$ and $F$, respectively, we obtain that $u_3+f=\underline{m}+g$, and so we can rewrite the equation above as 
  \[
 X_{U_3\oplus U_4}y^{\underline{m}-u_3}z^{\underline{n}-u_4}= \sum_{\substack{ 0\leq e\leq u_4\\  u_2< f\leq \underline{m}} }y^{e+f}z^{\underline{n}+\underline{m}  -(e+f)}\prod_{r=1}^nx_r^{-\langle e +f+u_3-\underline{m},\alpha_r\rangle -\langle \alpha_r,\underline{m}+u_4 -(e+f)\rangle}.\]
Now, because $\Phi(\underline{m}-u_3)=\underline{n}-u_4$, where $\Phi$ is the Coxeter transformation, it follows that $$\langle u_3-\underline{m},\alpha_r\rangle=-\langle\alpha_r,u_4-\underline{n}\rangle;$$ hence
$\langle e +f+u_3-\underline{m},\alpha_r\rangle +\langle \alpha_r,\underline{m}+u_4 -(e+f)\rangle= \langle e +f,\alpha_r\rangle +\langle \alpha_r,\underline{n}+\underline{m} -(e+f)\rangle$
and thus

\begin{equation}\label{eq:Uyz}
 X_{U_3\oplus U_4}y^{\underline{m}-u_3}z^{\underline{n}-u_4}= \sum_{\substack{ 0\leq e\leq u_4\\  u_2< f\leq \underline{m}} }y^{e+f}z^{\underline{n}+\underline{m}  -(e+f)}\prod_{r=1}^nx_r^{-\langle e +f,\alpha_r\rangle -\langle \alpha_r,\underline{n}+\underline{m} -(e+f)\rangle}.
 \end{equation}
 
  Notice that formulas on equations \eqref{eq:NM}, \eqref{eq:U} and \eqref{eq:Uyz} differ only on  where the sum is taken. Thus, in order to finish the proof, it suffices to show that the summation of \eqref{eq:U} and \eqref{eq:Uyz} gives the formula on equation \eqref{eq:NM}. For this purpose, we must reorganize the indexes of the summations on equations \eqref{eq:U} and \eqref{eq:Uyz}. Consider the portion of $X_{U_1 \oplus U_2}$ corresponding to $N \subsetneq E \subset U_1$. For each submodule $F$ of $U_2$, there exists an exact sequence $0\rightarrow E'\rightarrow E\oplus F\rightarrow F'\rightarrow 0$, which allows us to conclude that $e+f=e'+f'$. When rewriting the mentioned portion we get
\begin{eqnarray*}
&  &\displaystyle\sum_{\substack{ \underline{n}< e\leq u_1\\  0\leq f\leq u_2} }y^{e+f}z^{\underline{n}+\underline{m} -(e+f)}\prod_{r=1}^nx_r^{-\langle e +f,\alpha_r\rangle -\langle \alpha_r,\underline{n}+\underline{m} -(e+f)\rangle} =\\
& =&\displaystyle \sum_{\substack{ u_4< e'\leq \underline{n}\\  u_2< f'\leq\underline{ m}} }y^{e'+f'}z^{\underline{n}+\underline{m} -(e'+f')}\prod_{r=1}^nx_r^{-\langle e' +f',\alpha_r\rangle -\langle \alpha_r,\underline{n}+\underline{m} -(e'+f')\rangle}.
\end{eqnarray*}
% \[\begin{array}{lcl} \sum_{\substack{ n< e\leq u_1\\  0\leq f\leq u_2} }\prod_{r=1}^nx_r^{-\langle e +f,\alpha_r\rangle -\langle \alpha_r,\underline{n}+\underline{m} -(e+f)\rangle}y^{e+f}z^{\underline{n}+\underline{m} -(e+f)}
 %&=&\\
 %= \sum_{\substack{ u_4< e'\leq n\\  u_2< f'\leq m} }\prod_{r=1}^nx_r^{-\langle e' +f',\alpha_r\rangle -\langle \alpha_r,\underline{n}+\underline{m} -(e'+f')\rangle}y^{e'+f'}z^{\underline{n}+\underline{m} -(e'+f')} & &\end{array}.\]
By comparing where the sums are taken in equations \eqref{eq:U}, \eqref{eq:Uyz} and in the equation above, item $(a)$ follows. 
\medskip

For proving item $(b)$, since $M=P_i[1]$ and taking Figure \ref{fig:cmrg} into account, we may assume $U_1=P_{i'}[1]$ and $U_3=P_{i''}[1]$, from where we get $H^0B=U_2$ and $H^0B'=U_4$. Thus it suffices to prove that  \[X_Nx_i=x_{i'}X_{U_2}y^{\underline{n}-u_2}
 +x_{i''}X_{U_4}z^{\underline{n} -u_4}
.\] By Definition \ref{clustercharacter},
$\displaystyle X_{U_2}y^{\underline{n}-u_2}=\sum_{0\leq f\leq u_2}y^{f+\underline{n}-u_2}z^{u_2-f}\prod_{r=1}^nx_r^{-\langle f,\alpha_r\rangle -\langle \alpha_r,u_2-f\rangle}.$

\medskip

\noindent Now, for each indecomposable submodule $F$ of $U_2$, there exists an exact sequence $0\rightarrow E\rightarrow N\oplus F\rightarrow U_2\rightarrow 0$, where $U_4\subset E\subseteq N$, and then $e+u_2=\underline{n}+f$. By rewriting the above equation, we obtain that
$\displaystyle X_{U_2}y^{\underline{n}-u_2}=\sum_{u_4< e\leq \underline{n}}y^{e}z^{\underline{n}-e}\prod_{r=1}^nx_r^{-\langle e+u_2-\underline{n},\alpha_r\rangle -\langle \alpha_r,\underline{n}-e\rangle}.$
Since \[\langle \underline{n}-u_2,\alpha_r\rangle=\left\lbrace 
\begin{matrix} 1, & \mbox{if } r=i\\
               -1,&\mbox{if } r=i'\\
               0, &\mbox{otherwise},                                                                        

\end{matrix}\right.\]  
it follows that 
$\displaystyle X_{U_2}y^{\underline{n}-u_2}=x_ix_{i'}^{-1}\sum_{u_4< e\leq\underline{ n}}y^{e}z^{\underline{n}-e}\prod_{r=1}^nx_r^{-\langle e,\alpha_r\rangle -\langle \alpha_r,\underline{n}-e\rangle}.$

\medskip

\noindent On the other hand, $\displaystyle X_{U_4}z^{\underline{n} -u_4}=\sum_{0\leq e\leq u_4}y^{e}z^{\underline{n}-e}\prod_{r=1}^nx_r^{-\langle e,\alpha_r\rangle -\langle \alpha_r,u_4-e\rangle};$ because 
\[\langle \alpha_r,\underline{n}-u_4\rangle=\left\lbrace 
\begin{matrix} 1, & \mbox{if } r=i\\
               -1,&\mbox{if } r=i''\\
               0, &\mbox{otherwise},                                                                        

\end{matrix}\right.\] the above equation can be rewritten as \[X_{U_4}z^{\underline{n} -u_4}=x_ix_{i''}^{-1}\sum_{0\leq e\leq u_4}y^{e}z^{\underline{n}-e}\prod_{r=1}^nx_r^{-\langle e,\alpha_r\rangle -\langle \alpha_r,\underline{n}-e\rangle}.\] Therefore
$x_{i'}X_{U_2}y^{\underline{n}-u_2}
 +x_{i''}X_{U_4}z^{\underline{n} -u_4}= x_iX_N$,
as desired.
\end{proof}

Although we are not focusing on friezes and quantized Chebyshev polynomials, the presented definition of cluster character with coefficients (def. \ref{clustercharacter}) was based on results established for these concepts. In the next two corollaries, we recollect some particular cases of the multiplication formula obtained above, which will be used later and also generalize the tridiagonal matrix formula for c-friezes (\cite{BRS}) and quantized Chebyshev polynomials (\cite{Dupont}).
\medskip

Indexing the indecomposable objects of $\mathcal{C}$  by $(i,j)=\tau^{-i}P_j$ for $1\leq j\leq n$ and $0\leq i\leq n-j+1$, one can assume that $N=(i-\beta,j+\beta)$ and $M=(i+1,j+\alpha-1)$  where  $0\leq\beta\leq i$ and $i+j+\alpha\leq n+1$. In this case, using the notation established in Figure \ref{fig:cmrg}, $U_1= (i-\beta, j+\alpha+\beta) $, $U_2= (i+1,j-1) $, $U_3= (i+j+1,\alpha-1)$ and $U_4=(i-\beta,\beta).$

\begin{corol}\label{proofconja} Let $(i,j)=\tau^{-i}P_j$ where $1\leq j\leq n$ and $0\leq i\leq n-j$. Then
\begin{enumerate}
   \item[$(i)$] $X_{(i,j)}X_{(i+j,1)}=X_{(i,j+1)}+X_{(i,j-1)}z_{i+j}y_{i+j+1}$
   
   \item[$(ii)$]   $X_{(i-\beta,j+\beta)}X_{(i+1,j+\alpha-1)}=X_{(i+1,j-1)}X_{(i-\beta,j+\alpha+\beta)}+X_{(i-\beta,\beta)}X_{(i+j+1,\alpha-1)}z^{u}y^{v}$ where $u=\dim(i-\beta,j+\beta)-\dim(i-\beta,\beta)$, $v=\dim(i+1,j+\alpha-1)-\dim(i+j+1,\alpha-1)$ and $0\leq\beta\leq i$.
   
   \item[$(iii)$]  \[X_{(i,j)}=\left|\begin{array}{ccccc}
                      X_{(i,1)}&y_{i+2}&0&\cdots &0\\
                      z_{i+1}&X_{(i+1,1)}&y_{i+3}&\cdots &0\\
                      0      &z_{i+2}&\ddots&& \vdots\\
                      \vdots &       &       && y_{i+j} \\
                      0      &\cdots &       0&z_{i+j-1}&X_{(i+j-1,1)}
                      
\end{array}\right|\] 
\end{enumerate}
\end{corol}

\begin{proof}
By straightforward induction on $j$ and item $(i)$ the result in item $(iii)$ follows.  
\end{proof}

\begin{corol}\label{proofconjb} Let $(i,j)=\tau^{-i}P_j$. Then
\begin{enumerate}
   \item[$(i)$]   $x_jX_{(j-1,1)}=x_{j-1}y_j+x_{j+1}z_j$
   \item[$(ii)$]  $X_{(j,\beta -j)}x_{j+\alpha} = X_{(j+\alpha, \beta-j-\alpha)}x_jy^{\dim(j,\alpha)} + X_{(j,\alpha-1)}x_{\beta+1}z^{\dim (j+\alpha-1,\beta-j-\alpha+1)}$
 
   \item[$(iii)$]  \[z_1\cdots z_{j-1}x_j=\left|\begin{array}{ccccc}
                      x_1&y_1&0&\cdots &0\\
                      1&X_{(0,1)}&y_2&\cdots &0\\
                      0      &z_1&\ddots&& \vdots\\
                      \vdots &       &       && y_{j-1} \\
                      0      &\cdots &       0&z_{j-2}&X_{(j-2,1)}
                      
\end{array}\right|\]  
\end{enumerate}
\end{corol} 
\begin{proof}
By straightforward induction on $j$ and item $(i)$ the result in item $(iii)$ follows.  
\end{proof}

\begin{rmk}
Observe that the latter items of the two above corollaries implies that, for Dynkin case $\A$, any cluster character is a polynomial in the cluster characters of the indecomposable simple objects.  
\end{rmk}
\medskip

\subsection{The mutation class of $\widehat{\mathbb{A}}_n$} This section is devoted to describe the mutation class of the biframed quiver $\widehat{\mathbb{A}}_n$ in terms of the cluster tilting objects in the cluster category associated to the principal part of  $\widehat{\mathbb{A}}_n$. For this purpose the following definition is needed, where an ice quiver is constructed from a cluster tilting object.

\begin{defi}\label{Q_T} Let $T$ be a cluster tilting object in  $\mathcal{C}=\mathcal{C}_Q$, where $Q$ is a Dynkin quiver. The ice quiver $Q_T$ of type $(n,3n)$ associated to $T$ is defined by:
\begin{enumerate}
        \item The principal part of $Q_T$ is the opposite quiver  of the endomorphism algebra $\mathrm{End}_{\mathcal{C}}T$.
        \item Let $T'_r$ be the other complement of the almost complete basic tilting object $T\slash T_r=\oplus_{j\neq r}T_j$ and consider the triangles 
        
          \[T_r\rightarrow B\rightarrow T'_r\rightarrow T_r[1]\  \mbox{and}\ T'_r\rightarrow B'\rightarrow T_r\rightarrow T'_r[1]\] as in subsection \ref{mutationT}.

        \begin{enumerate}
            \item[$(a)$]
             Let  $H^0(B)=B_1\oplus B_2$ where $B_1$ is the direct sum of all successors of $T_r$ in $\textrm{mod} kQ$ and $B_2$ is the complement of $B_1$ to $H^0(B)$.
             Let  $H^0(B')=B'_1\oplus B'_2$ where $B'_1$ is the direct sum of all successors of $T'_r$ in $\textrm{mod} kQ$ and $B'_2$ is the complement of $B'_1$ to $H^0(B')$.  If $T_r$ and $T'_r$ are modules and $T_r$ is a predecessor (resp. successor) of $T'_r$ in $\mathrm{mod} kQ$ then the number of arrows from $r$ to $n+j$ (resp. from $n+j$ to $r$) is given by the $j^{th}$ coordinate of 
                  $\dim T'_r-\dim B'_1$ (resp. $\dim T_r-\dim B_1)$ and the number of arrows from $r$ to $2n+j$ (resp. from $2n+j$ to $r$)  is given by the $j^{th}$ coordinate of $\dim T_r-\dim B'_2$ (resp. $\dim T'_r-\dim B_2$).
             \item[$(b)$]   Suppose that  $T'_r=P[1]$ (resp. $T_r=P[1]$) for some projective module $P$. Given $j$, the number of arrows from $n+j$ to $r$  (resp. from $r$ to $n+j$) is the $j^{th}$ coordinate of $\dim T_r-\dim H^0(B)$ (resp. $\dim T'_r-\dim H^0(B')$) and  the number of arrows from $r$ to $2n+j$ (resp. from $2n+j$ to $r$) is the $j^{th}$ coordinate of $\dim T_r -\dim H^0(B')$ (resp. $\dim T'_r-\dim H^0(B)$).    
        \end{enumerate}
          
\end{enumerate} 

\end{defi}

Notice that if $T=\oplus_{i=1}^n P_i[1]$ then $Q_T$ is the biframed quiver associated to the ordinary quiver of $\mathrm{End}_{\mathcal{C}}kQ$.

%\begin{exmp} exemplo $Q_T$ do $\D_n$

%\end{exmp}

\begin{exmp}\label{exmp:qtemr} For a given cluster tilting object $T$ of $\CC_{\A_n}$, we describe the ice quiver $Q_T$  associated to $T$ at an arbitrary vertex $r$. Let $T_r$ be an indecomposable summand of $T$ and consider the exchange triangles 
          \[T_r\rightarrow B\rightarrow T'_r\rightarrow T_r[1]\  \mbox{and}\ T'_r\rightarrow B'\rightarrow T_r\rightarrow T'_r[1].\] As mentioned in subsection \ref{sec:multiAn}, these objects are disposed as in Figure \ref{fig:cmrg}, where $T_r$ is either $N$ or $M$ (and $T'_r$ is either $M$ or $N$, respectively). One can assume that $T_r=N$ since the other case is obtained by applying mutation. Notice that the principal part of $Q_T$ at vertex $r$ is 
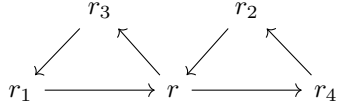
\begin{center}
\begin{tikzcd}[column sep=small]
\label{fig:pp}
& r_3 \ar[dl] & & r_2 \ar[dl] & \\
r_1 \ar[rr] && r \ar[rr] \ar[ul] && r_4 \ar[ul] 
\end{tikzcd}
\end{center}
 where $r_j$ is the vertex associated to $U_j$. Now, in order to compute the arrows between frozen vertices and $r$, according to Definition \ref{Q_T}, we need to consider the following two cases: 
\begin{enumerate}
\item $T_r$ and $T'_r$ are $kQ$-modules  
\item $T_r$ or $T'_r$ is a shift of an indecomposable projective object $P$. 
\end{enumerate}
Since in both cases  $Q_T$ can be obtained in a similar way, we assume $(1)$. Following the notation of Corollary \ref{proofconja},  by indexing the objects of $\CC$ by $(i,j)=\tau^{-i}P_j$, we have that
\[T_r=(i-\beta,j+\beta),\   \  T'_r=(i+1,j+\alpha-1),\]   $B=U_1\oplus U_2$ and $B'=U_3\oplus U_4$, where  $U_1= (i-\beta, j+\alpha+\beta) $, $U_2= (i+1,j-1) $, $U_3= (i+j+1,\alpha-1)$ and $U_4=(i-\beta,\beta).$  Then, using Definition \ref{Q_T}$(a)$, in $Q_T$ there are arrows from $r$ to the frozen vertices $n+i+2,\cdots, n+i+j+1$ and $2n+i+1,\cdots, 2n+i+j$, because \[\dim T'_r-\dim B'_1=\dim T'_r-\dim U_3=\dim(i+1,j)\]
and  
 
 \[\dim T_r-\dim B'_2=\dim T_r -\dim U_4=\dim(i ,j ),\] respectively. Therefore $Q_T$ has the following shape at vertex $r$ 
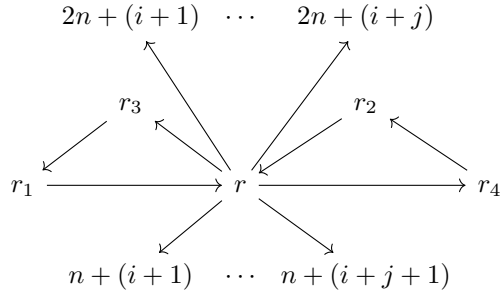
\begin{center}

\begin{tikzcd}[column sep=0.01in]
\label{fig:QUEMR}
&& 2n+(i+1) & \cdots & 2n+(i+j) && \\
& & r_3 \ar[dll] && r_2 \ar[ld]&& \\
r_1 \ar[rrr] && &r \ar[uul] \ar[uur] \ar[ul]  \ar[dl]  \ar[dr]  \ar[rrr] &&& r_4 \ar[llu] \\
&& n+(i+1) & \cdots & n+(i+j+1) &&
\end{tikzcd}

\end{center}

\end{exmp}

\begin{rmk}\label{rmk:lazy} For proving the next results of this work it will  be sometimes necessary to consider the cases whether $T_r$ is $N$ or $M$  (as in Figure \ref{fig:cmrg}) and also whether $T_r$ (or $T'_r$) is a $k\A_n$-module or not, as in the example above. Having in mind the similarity of those cases, we focus our attention to $T_r=N$ and when $T_r$ and $T'_r$ are $k\A_n$-modules. 
\end{rmk}

It is desirable that  mutations of quivers and cluster tilting objects are compatible with  the correspondence defined  in \ref{Q_T}. In the following proposition one prove this result for the  Dynkin case $\A_n$.
\begin{prop}\label{commutes} 
	Let $T$ be a cluster tilting object in $\CC_{\A_n}$. The map $T\mapsto Q_T$ commutes with mutation, that is, $\mu_r(Q_T)=Q_{\mu_r(T)}$.
\end{prop}
\begin{proof}% Because of the definition of quiver mutation at some vertex $r$, in order to compute $\mu_r(Q_T)$,  we focus our attention to vertices $r,i$ and $j$ such that there is a linear oriented path $i\rightarrow r\rightarrow j$.  
Let $T$ be a cluster tilting object in $\CC$. 
It is known that the result holds true for the principal part of $Q_T$ (\cite{BMR2}, Theorem 5.1). We then focus our attention to arrows between $r$ and frozen vertices and possible new arrows created by mutation at $r$. Taking into account Remark \ref{rmk:lazy}  $Q_T$ has the configuration presented in Example \ref{exmp:qtemr}.   Now, because mutation of cluster tilting object is an involution we obtain $(\mu_r(T))_r=T'_r$ and $(\mu_r(T))'_r=T_r$. As a consequence, by Definition \ref{Q_T},  in $Q_{\mu_r(T)}$ there are arrows from the frozen vertices  $n+i+2,\cdots, n+i+j+1$ and $2n+i+1,\cdots, 2n+i+j$ to $r$. Therefore, the arrows at vertex $r$ in $Q_T$ and in $Q_{\mu_r(T)}$ are opposite to one another, which is compatible with quiver mutation.
 
  To finish the proof, it suffices to verify that for each oriented path  $r_1\rightarrow r\rightarrow s $ (or $ s\rightarrow r\rightarrow r_1$), where $s$ is a frozen and $r_1$ is a non frozen vertex, the quivers $\mu_r(Q_T)$ and $Q_{\mu_r(T)}$ coincide at vertex $r_1$. Notice that $T_{r_1}$ is a direct summand of $B$, that is, $T_{r_1}=(i+1,j-1)$ or $T_{r_1}=(i-\beta, j+\alpha+\beta)$. Assume that $T_{r_1}=(i+1,j-1)$. The other case can be treated similarly. Consider the triangles 
  
  \[T_{r_1} \rightarrow B\rightarrow T'_{r_1}\rightarrow T_{r_1}[1]\  \mbox{and}\ T'_{r_1}\rightarrow B'\rightarrow T_{r_1}\rightarrow T'_{r_1}[1],\mbox{where}\]
   \[T'_{r_1}=(i-\beta,j+\beta -\ell), B = (i+j+1-\ell,\ell-1)\oplus(i-\beta,\beta) \mbox{ and } B'= (i+1,j-\ell-1)\oplus (i-\beta, j+\beta).\]
    According to Definition \ref{Q_T}, in $Q_T$ we have the following configuration between frozen vertices and $r_1$ 
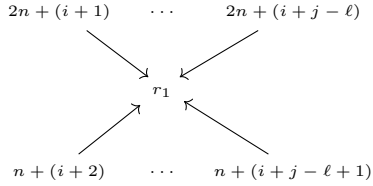
\begin{center}
\begin{tikzcd}[column sep=tiny, font=\tiny]
\label{fig:r_1 em Q_T}
2n+(i+1) \ar[rd] & \cdots & 2n+(i+j-\ell) \ar[dl] \\
& r_1 & \\
n+(i+2) \ar[ur] & \cdots & n+(i+j-\ell +1) \ar[ul]
\end{tikzcd}
\end{center}
. Since there are arrows from $r$ to  the frozen vertices $n+i+2,\cdots, n+i+j+1$, $2n+i+1,\cdots, 2n+i+j$ and an arrow from $r_1$ to $r$, after mutating $Q_T$ at $r$, will be created arrows from $r_1$ to all such frozen vertices. However, by definition of quiver mutation, the two cycles have to be removed and then $\mu_r(Q_T)$ has the following configuration between frozen vertices and $r_1$: 
    \begin{center}
\begin{tikzcd}[column sep=tiny, font=\tiny]
\label{fig:r_1 em mu}
2n+(i+j-\ell +1) & \cdots & 2n+(i+j) \\
& r_1\ar[ul] \ar[ur] \ar[dl] \ar[dr ]& \\
n+(i+j-\ell +2) & \cdots & n+(i+j+1)
\end{tikzcd}
\end{center}

    Now for computing $Q_{\mu_r(T)}$ at vertex $r_1$, consider the triangles
        \[\mu_r(T)_{r_1} \rightarrow B\rightarrow (\mu_r(T))'_{r_1}\rightarrow \mu_r(T)_{r_1}[1]\  \mbox{and}\ (\mu_r(T))'_{r_1}\rightarrow B'\rightarrow \mu_r(T)_{r_1}\rightarrow (\mu_r(T))'_{r_1}[1].\] 
     Observe that $\mu_r(T)_{r_1}=T_{r_1}$, $B=(i+1,j+\alpha-1)\oplus(i+j+1-\ell,\ell-1)$, $B'=(i+j+1,\alpha-1)\oplus(i+1,j-1-\ell)$ and therefore  $(\mu_r(T))'_{r_1}=(i+j+1-\ell,\alpha-1+\ell)$. Again, by Definition \ref{Q_T}, when looking at frozen vertices and $r_1$, $Q_{\mu_r(T)}$ coincides with $\mu_r(Q_T)$, which finishes the proof.

\end{proof}

In case $\A_n$, it follows from the previous proposition that  all quivers of $\mathcal{M}(\widehat{\A})$ can be obtained  from cluster tilting objects. In fact, one shall prove in the next theorem that the correspondence $T\mapsto Q_T$ is a bijection.  

\begin{theorem}\label{teo:bij} 
	Let $\widehat{\mathbb{A}}_n$ be the biframed quiver of type $\mathbb{A}_n$ and let $\mathcal{C}=\mathcal{C}_{\mathbb{A}_n}$ be the cluster category associated to the principal part of $\widehat{\mathbb{A}}_n$. There exists a bijection between  the cluster tilting objects of $\mathcal{C}$ and the mutation class of $\widehat{\mathbb{A}}_n$ given by $T\mapsto Q_T$.  
\end{theorem}
\begin{proof} Let $Q$ be a quiver in $\mathcal{M}(\widehat{\A})$, that is, the mutation class of the biframed quiver $\widehat{\A}$. In this situation, there exists a sequence of mutations $\mu$ such that $\mu(Q)=\widehat{\A}_n$. On the other hand, it follows by  Definition \ref{Q_T} that $Q_{k\widehat{\A}_n[1]}=\widehat{\A}_n$ and therefore $\mu(Q)=Q_{k\widehat{\A}_n[1]}$. According to Proposition \ref{commutes},  $Q=Q_{\mu^{-1}(k\widehat{\A}_n[1])}$, which entails that there is a surjective map from the set of cluster tilting objects of $\CC$ to $\mathcal{M}(\widehat{\A})$ given by $T\mapsto Q_T$.

For proving that this map is injective it suffices to verify that the quivers $\widehat{\A}_n$ and $Q_{\tau^i(k\A_n[1])}$  are not isomorphic for $1\leq i\leq n+2$. Indeed,  let $T$ and $T'$ be cluster tilting objects of $\CC$ such that $Q_T=Q_{T'}$. In particular, the principal parts of $Q_T$ and $Q_{T'}$ are the same, which allows us to conclude that $T'=\tau^i T$ for some $0\leq i\leq n+2$. Since the cluster tilting graph is connected, \cite{BMRRT}, there exists a sequence of cluster tilting mutations $\mu$ such that $\mu(T)=k\A_n[1]$ and then

\[\mu(T')=\mu\tau^i(T)=\tau^i\mu(T)=\tau^i(k\A_n[1]).\] Now, from Proposition \ref{commutes}, it follows that

\[Q_{\mu(T)}=\mu(Q_T)=\mu(Q_{T'})=Q_{\mu (T')},\] that is  
       
        \[\widehat{\A}_n=Q_{k\A_n[1]}=Q_{\tau^i(k\A_n[1])}.\] Because of our claim, $i=0$ and $T=T'$.  Then we focus our efforts to prove that this claim holds true. First, lets do the comparison of the quivers $\widehat{\A}_n$ and $Q_{\tau^i(k\A_n[1])}$ with $1\leq i < n$. For this purpose we look at vertex $n$ of $Q_{\tau^i(k\A_n[1])}$, which is associated to the $n^{th}$ indecomposable summand  $\tau^i(P_n[1])$  of the cluster tilting object $\tau^i(k\A_n[1])$.
Notice that we have the following triangles

\[S_{n-i+1} \rightarrow 0 \rightarrow S_{n-i} \rightarrow S_{n-i+1}[1]\] and \[S_{n-i} \rightarrow \left(\begin{smallmatrix} n-i+1 \\ n-i \end{smallmatrix}\right) \rightarrow S_{n-i+1} \rightarrow S_{n-i}[1].\] 
\\
\noindent Naming $B = B_1 = B_2 = 0$, $B'=B'_1 = \left(\begin{smallmatrix} n-i+1 \\ n-i \end{smallmatrix}\right)$ and $B'_2 = 0$, it follows from Definition \ref{Q_T}$(a)$ that, at vertex $n$, the ice quiver  $Q_{\tau^i(k\A_n[1])}$ has the configuration 

\begin{center}
\begin{tikzcd}[column sep=tiny]
3n-i \ar[d] & \\
n \ar[r] & n-1 \\
2n-i+1 \ar[u] & 
\end{tikzcd}
\end{center}
Thus $\widehat{\A}_n$ and $Q_{\tau^i(k\A_n[1])}$ are not isomorphic for  $1\leq i < n$. The remaining cases, $i=n, n+1, n+2$, can be obtained in a similar way, looking at  vertices $n-1$ or $n$.  
\end{proof} 

\subsection*{The main theorem} In this section, we establish the connection between cluster algebras with coefficients associated to the biframed quiver of type $\A_n$ and the cluster character presented in Definition \ref{clustercharacter}. For this purpose, the multiplication formula presented in \ref{teo:multf} will be of great importance since it plays the same role as the exchange relation for cluster algebras.

\begin{theorem}\label{main theorem} 
  The cluster variables of a cluster algebra with  coefficients associated to the biframed quiver $\widehat{Q}$ of type $\mathbb{A}_n$ are precisely the characters $X_V$ when $V$ runs over the indecomposable rigid objects in $\mathcal{C}_{\A_n}$.
  \end{theorem}
  \begin{proof}  Observe that, by definition, the character of $P_r[1]$ is the initial cluster variable $x_r$. Since any cluster tilting object in $\CC$ can  be obtained from $k\A_n[1]$ by a finite number $\delta$ of mutations (\cite{BMRRT}, Proposition 3.5), we  prove by induction on $\delta$ that the character of each indecomposable direct summand of a tilting object $T$ of $\CC$ is a  cluster variable.

Let $(\widehat{\A}_n, \mathcal{X})$ be the initial seed of $\mathcal{A}$. For $\delta=1$, using what it was settled in subsection \ref{Biframed},  the new cluster variable obtained by mutating the initial seed at $r$ is given by the following relation

\[x_rx_r'=z_rx_{r+1}+x_{r-1}y_r.  \]  
 Thus, using Corollary \ref{proofconjb}$(i)$, we get that $z_rx_{r+1}+x_{r-1}y_r = x_rX_{(r-1,1)}$, that is, the character  $X_{(r-1,1)}=x_r'$ is a cluster variable and the new seed $(\mu_r(\widehat{\A}_n),\mathcal{X}')$ is such that each cluster variable of $\mathcal{X}'$ is a character of an indecomposable summand of $\mu_r(k\A_n[1])$, since $P_r[1]$ is replaced by $S_r=(r-1,1)$ after cluster tilting mutation at $r$.

 Let $\mu$ be a sequence of $\delta$ mutations and $\mu(\widehat{\A}_n, \mathcal{X})=(Q,\mathcal{U})$. Consider a cluster tilting object $T$ such that $Q_T=Q$.    Suppose by induction that each cluster variable $u_r$ in $\mathcal{U}$  is the character $X_{T_r},$ where $T_r$ is an indecomposable summand of $T$. According to Remark \ref{rmk:lazy} and Example \ref{exmp:qtemr}, $Q$ has the following configuration at vertex $r$:

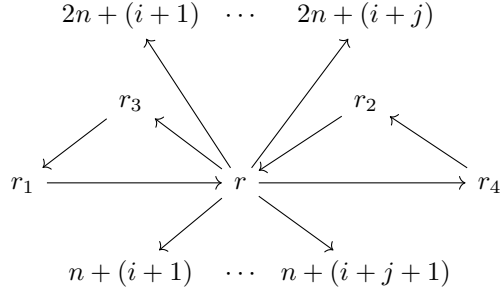
\begin{center}
\begin{tikzcd}[column sep=0.01in]
%\label{fig:QUEMR}
&& 2n+(i+1) & \cdots & 2n+(i+j) && \\
& & r_3 \ar[dll] && r_2 \ar[ld]&& \\
r_1 \ar[rrr] && &r \ar[uul] \ar[uur] \ar[ul]  \ar[dl]  \ar[dr]  \ar[rrr] &&& r_4 \ar[llu] \\
&& n+(i+1) & \cdots & n+(i+j+1) &&
\end{tikzcd}
\end{center}

Hence mutation at $r$ gives rise to the relation

   \[u_{r}u_{r}'=u_{r_1}u_{r_2}+u_{r_3}u_{r_4}z^{\dim(i,j)}y^{\dim(i,j+1)}\] and then, by the induction hypothesis, 
   \[X_{T_{r}}u_{r}'=X_{T_{r_1}}X_{T_{r_2}}+X_{T_{r_3}}X_{T_{r_4}}z^{\dim(i,j)}y^{\dim(i,j+1)}\]
  It follows from the correspondence $T\mapsto Q_T$ given in Theorem \ref{teo:bij} and the discussion in  Example \ref{exmp:qtemr} that \[T_{r_1}\oplus T_{r_2}=(i+1,j-1)\oplus (i-\beta,j+\alpha+\beta) \mbox{ and } T_{r_3}\oplus T_{r_4}= (i+j+1,\alpha-1) \oplus(i-\beta,\beta).\] By applying Corollary \ref{proofconja}$(ii)$ we conclude that $u'_r=X_{(i+1,j+\alpha-1)}=X_{T_r'}$, which finishes the induction. Since any cluster variable lies in a seed which can be obtained from the initial seed by a finite number of mutations the result follows.    
\end{proof}
\medskip

\begin{corol}
 The cluster variables of the cluster algebra $\mathcal{A}(\textbf{x},\widehat{\A}_n)$ are polynomials over $\Q[\textbf{y},\textbf{z}^{\pm 1}]$ in the characters of the simple modules and the character of the shift of the simple projective, that is, \[\mathcal{A}(\textbf{x},\widehat{\A}_n)=\Q[\textbf{y},\textbf{z}^{\pm 1}][x_1,X_{S_1},\cdots,X_{S_n}].\]  

\end{corol}
\begin{proof}
It follows directly from Theorem \ref{main theorem} and Corollaries \ref{proofconja}$(iii)$ and \ref{proofconjb}$(iii)$.  
\end{proof}

It is well known that any cluster seed is uniquely determined by its cluster \cite{cluster2}. In the next corollary we will prove that any cluster seed of a cluster algebra associated to the biframed quiver $\widehat{\A}_n$ is also uniquely determined by its quiver. 

\begin{corol} 
\end{corol}
\begin{enumerate}
\item[$(a)$] Any seed of the cluster algebra $\mathcal{A}_{\widehat{\A}_n}$ is uniquely determined by its quiver.
\item[$(b)$] Cluster tilting objects corresponds to clusters. 
\end{enumerate}
\begin{proof}
Observe that, according to the proof of Theorem \ref{main theorem}, if $(Q,\mathcal{U})$ is a seed of the cluster algebra $\mathcal{A}_{\widehat{\A}_n}$ and $T$ is the cluster tilting object associated to $Q$ then the mutable cluster variables $u_r$ in $\mathcal{U}$ are given by $X_{T_r}$, for $r=1,\cdots,n$. 
\end{proof}

%\newpage

\section{A Multiplication Formula for Dynkin Case}

As it was seen in Dynkin case $\A$, the multiplication formula (Theorem \ref{teo:multf}) for cluster characters with coefficients plays a special role in giving a connection with cluster algebras with coefficients associated to a biframed quiver. In this section we present this same formula for the others Dynkin types, but only when one of the exchange triangles is almost split; because of this restriction to triangles, we could not obtain the correspondence between cluster tilting objects and clusters (as done for case $\mathbb{A}$). However, it is still possible to prove that all cluster characters are cluster variables. 

\begin{theorem}\label{auslander} Let $Q$ be a quiver of Dynkin type and $H^0=\Hom(\oplus_i P_i,\_):\mathcal{C}_Q\rightarrow\mathrm{mod}kQ$. Let $M$ and $N$ be indecomposable objects in cluster category $\mathcal{C}_Q$ such that $\Ext^1_{\mathcal{C}_Q}(M,N)$ is one-dimensional. 
\begin{itemize}
 \item[$(a)$]
    If $N=\tau M$  where $M$ is a non-projective module, then   \[X_NX_M=X_{B}+z^{\underline{\dim}N}y^{\underline{\dim}M};\]

\item[$(b)$] If $M=P_i[1]$ for an $i\in Q_0$ and $N=\tau M$, then \[X_NX_M=X_{B}y^{\underline{\dim}N-\underline{\dim}H^0B}
 +z^{\underline{\dim}N };
\]

\item[$(c)$] If $N=P_i[1]$ and $M=P_i$ for an $i\in Q_0$, then \[X_NX_M=X_{B}z^{\underline{\dim}M-\underline{\dim}H^0B}
 +y^{\underline{\dim}M };
\]

\end{itemize}  where $B$ is the unique object such that there exists an almost split triangle  $N\rightarrow B\rightarrow M\rightarrow N[1].$

\end{theorem}
\begin{proof}
We only prove $(b)$ since $(c)$ can be obtained in a similar way, and $(a)$ is analogous to Proposition 2.2 of \cite{Dupont}. 

Let $d = \underline{\dim} N$ and $h=\underline{\dim}H^0(B)$. According to Definition \ref{clustercharacter} and Euler form, 

\[
 \begin{array}{lcl} X_MX_N&=& x_i\sum_{0\leq e\leq d} y^ez^{d-e}\chi(Gr_e(N))\prod_{r=1}^nx_r^{-\langle e,\alpha_r\rangle -\langle \alpha_r,d-e\rangle}\\
                       &=& x_i\left(\frac{1}{x^d}\sum_{0\leq e\leq d} y^ez^{d-e}\chi(Gr_e(N))\prod_{r=1}^nx_r^{\sum_{j\rightarrow r} e_j +\sum_{r\rightarrow j}d_j-e_j}\right)\\
                       &=&\frac{1}{x^{h}} \left(\sum_{0< e\leq d} y^ez^{d-e}\chi(Gr_e(N))\prod_{r=1}^nx_r^{\sum_{j\rightarrow r} e_j +\sum_{r\rightarrow j}d_j-e_j}\right)+z^d\\
                                  &=&\frac{y^{d-h}}{x^{h}} \left(\sum_{0< e\leq d} y^{e+h -d}z^{d-e}\chi(Gr_e(N))\prod_{r=1}^nx_r^{\sum_{j\rightarrow r} e_j +\sum_{r\rightarrow j}d_j-e_j}\right)+z^d.\\\end{array}\]
  Taking $f=e+h-d$ and since in any Dynkin case it is possible to find a submodule of $N$ with dimension vector $d-h$, we can rewrite the above equation by
                                  
 \[\begin{array}{lcl}
                                  X_MX_N&=&\displaystyle y^{d-h} \left(\frac{1}{x^{h}}\sum_{0\leq f\leq h} y^{f}z^{h-f}\chi(Gr_f(H^0(B)))\prod_{r=1}^nx_r^{\sum_{j\rightarrow r} f_j+d_j-h_j +\sum_{r\rightarrow j}h_j-f_j}\right)+z^d\\
                                  &=&\displaystyle y^{d-h} x_{i-1}\left(\frac{1}{x^{h}}\sum_{0\leq f\leq h} y^{f}z^{h-f}\chi(Gr_f(H^0(B)))\prod_{r=1}^nx_r^{\sum_{j\rightarrow r} f_j +\sum_{r\rightarrow j}h_j-f_j}\right)+z^d\\
                                  &=& y^{d-h}x_{i-1}X_{H^0(B)} + z^d\\
                                  &=&X_By^{d-h}+z^d.\end{array}\]
\end{proof}

%\begin{conj}\label{conj} Let $H^0=\Hom_{\mathcal{C}}(\oplus_i P_i,\_):\mathcal{C}\rightarrow\mathrm{mod}kQ$ and let $M$ and $N$ be indecomposable objects of the category $\mathcal{C}_Q$ such that $\Ext^1_{\mathcal{C}}(M,N)$ is one-dimensional. 
%\begin{itemize}
% \item[$a)$]
%   If $N$ and $M$ are 	 modules such that $\Hom_{kQ}(M,N)=0$, then there exists modules $B'_1,B'_2$ such that $H^0(B')=B_1\oplus B_2$ and  \[X_NX_M=X_{B}+X_{B'}z^{\underline{\dim}N-\underline{\dim}B'_2}y^{\underline{\dim}M-\underline{\dim}B'_1}.\]
%
%
%\item[$b)$] If $M=P_i[1]$ for an $i\in Q_0$ and $N$ is an 	indecomposable module  \[X_NX_M=X_{B}y^{\underline{\dim}N-\underline{\dim}H^0B}
% +X_{B'}z^{\underline{\dim}N -\underline{\dim}H^0B'}
%\]
%
%\end{itemize}  where $B$ and $B'$ are the unique objects such that there exist non split triangles  \[N\rightarrow B\rightarrow M\rightarrow N[1]\  \mbox{and}\ M\rightarrow B'\rightarrow N%%\rightarrow M[1].\]
%
%\end{conj}

\medskip

As a consequence of Theorem \ref{auslander}, we  prove in the next result that any cluster character with coefficients is a cluster variable of a cluster algebra associated to a biframed quiver of Dynkin type. 

\begin{theorem}
  
Let $Q$ be a quiver of Dynkin type.  The characters $X_V$ when $V$ runs over the indecomposable rigid objects in $\mathcal{C}_Q$ are cluster variables of a cluster algebra with  coefficients associated to the biframed quiver $\widehat{Q}$.
\end{theorem}
\begin{proof}

We only prove case $Q=\mathbb{D}_n$ since case ${\mathbb A}_n$ is Theorem \ref{main theorem}, and the others can be obtained in a similar way. We point out that the proof presented here is inspired by Caldero-Chapoton's proof for coefficient-free case. Let $(\widehat{Q},x)$ be the initial cluster seed. First notice that $X_{P_i[1]}$ is the initial cluster variable $x_i$, for $i=1,\cdots,n$. It follows from Theorem \ref{auslander}$(b)$ that

\[x_1X_{ P_1}=x_{2}z_1+y_1.\]
On the  other hand, by mutating the initial seed at $1$, one can obtain that
\[x'_1x_1=x_{2}z_1+y_1.\] In this case, $x'_1=X_{ P_1}$ and therefore $X_{ P_1}$ is a  cluster variable. Suppose by induction on $1\leq r\leq n-4$ that in the seed $\mu_r\mu_{r-1}\cdots\mu_1(\widehat{Q},x)$, $x'_i=X_{P_i}$  for each $1\leq i\leq r$. Since, at vertex $r+1$,  $\mu_r\mu_{r-1}\cdots\mu_1(\widehat{Q},x)$ has the configuration

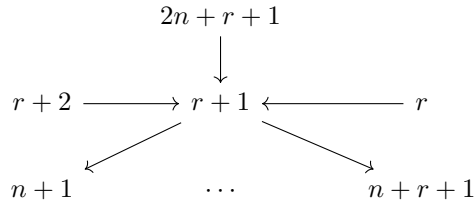
\begin{center}
\begin{tikzcd}
  &   2n+r+1\ar[d] & \\
 r+2\ar[r] & r+1 \ar[dl] \ar[dr]& r\ar[l] \\
 n+1& \cdots &n+r+1 
\end{tikzcd}
\end{center}
 the cluster variable $x'_{r+1}$ is obtained by the exchange relation 
 
\[x'_{r+1}x_{r+1}= x_{r+2}x'_rz_{r+1}+\prod_{t=1}^{r+1}y_{t}.\] Now, by Theorem \ref{auslander}$(c)$ 
	\[X_{P_{r+1}[1]}X_{P_{r+1}}=X_{P_r\oplus P_{r+2}[1]}z^{\dim{P_{r+1}}-\dim{P_r}}+y^{\dim{P_{r+1}}}.\] 
Recalling that $X_{P_{r+1}[1]}=x_{r+1}$, $X_{P_{r+2}[1]}=x_{r+2}$ and using induction hypothesis to conclude that $X_{P_r}=x'_r$ we get

 \[x_{r+1}X_{P_{r+1}}=x'_rx_{r+2}z^{\dim{P_{r+1}}-\dim{P_r}}+y^{\dim{P_{r+1}}}= x'_rx_{r+2}z_{r+1}+\prod_{t=1}^{r+1}y_{t}= x_{r+1}x'_{r+1},\] from where we have $X_{P_{r+1}}=x'_{r+1}.$  By applying the sequence of mutation $\mu_n\mu_{n-1}\mu_{n-2}$  to $\mu_{n-3}\cdots\mu_1(\widehat{Q},x)$ and comparing the exchange relation obtained in each step with multiplication formula given by Theorem \ref{auslander}$(c)$, as done above,  it is possible to prove that $X_{P_{n-2}}$, $X_{P_{n-1}}$ and $X_{P_n}$ are also cluster variables.   

 Let $\mu$ be the sequence of mutation $\mu_n\mu_{n-1}\cdots\mu_1$. According to the above considerations, in the seed $\mu(\widehat{Q},x) $ we have that $x'_r=X_{P_r}$. Given $1\leq k<n-2$, denote by $u_r$ the cluster variable  associated to the vertex $r$ of $\mu^k(\widehat{Q})$  in the seed $\mu^k(\widehat{Q},x)$.  Suppose by induction on $k$ that  $u_r=X_{\tau^{-k+1}P_r}$, $1\leq r\leq n$.	It follows from Theorem \ref{auslander}$(a)$ that
 
 		\[X_{\tau^{-k+1}P_1}X_{\tau^{-k}P_1}=X_{\tau^{-k+1}P_2}+z^{\underline{\dim}\tau^{-k+1}P_1}y^{\underline{\dim}\tau^{-k}P_1}.\]
On the  other hand, observing that the quiver $\mu^k(\widehat{Q})$ at vertex $r$, for $1\leq r\leq n-k-2$, is 
\[\begin{tikzcd}
& 2n+r+k-1 & \\
r+1 \ar[r] & r\ar[u] \ar[d] \ar[r] & r-1 \\
& n+r+k &
\end{tikzcd}
\]
 by mutating the seed $\mu^k(\widehat{Q},x)$ at $1$, one can obtain that
	
	\[u_1u_1'=u_2+z_{k}y_{k+1}.\]
By induction hypothesis and the previous equation, it follows that

	\[\begin{array}{lcl}
		X_{\tau^{-k+1}P_1}u_1'&=& X_{\tau^{-k+1}P_2}+z_{k}y_{k+1}\\ 		
			&=&X_{\tau^{-k+1}P_2}+z^{\underline{\dim}\tau^{-k+1}P_1}y^{\underline{\dim}\tau^{-k}P_1}\\
	  		&=&X_{\tau^{-k+1}P_1}X_{\tau^{-k}P_1}.
	\end{array}\]
In this case, $u_1'= X_{\tau^{-k}P_1}$ and therefore $X_{\tau^{-k}P_1}$ is a cluster variable.  By applying the sequence of mutation $\mu_n\mu_{n-1}\cdots\mu_2$  to $\mu_1(\widehat{Q},x)$ and comparing the exchange relation obtained in each step with multiplication formula given by Theorem \ref{auslander}$(a)$, it is possible to conclude by induction that $X_{\tau^{-k}P_r}$, where $2\leq r\leq n$, are also cluster variables. 

Because of the shape of the quiver $\mu_r\mu_{r-1}\cdots\mu_1\mu^k(\widehat{Q})$ at vertex $r+1$, it is necessary to consider four cases: induction on $1\leq r\leq n-k-2$,  induction on $n-k-1\leq r\leq n-2$, $r=n-1$ and $r=n$.

Suppose by induction on $1\leq r< n-k-2$ that in the seed $\mu_r\mu_{r-1}\cdots\mu_1\mu^k(\widehat{Q},x)$, $u_i'=X_{\tau^{-k}P_i}$ for each $1\leq i\leq r$. Since, at vertex $r+1$,  $\mu_r\mu_{r-1}\cdots\mu_1\mu^k(\widehat{Q})$ has the configuration

	\[
\begin{tikzcd}
2n+k & \cdots & 2n+k+r \\
r+2 \ar[r] & r+1 \ar[ul] \ar[ur] \ar[dl] \ar[dr] & r \ar[l] \\
n+k+1 & \cdots & n+k+r+1
\end{tikzcd}
\]
the cluster variable $u_{r+1}'$ is obtained by the exchange relation 

	\[ u_{r+1}u_{r+1}'=u_{r+2}u_r'+\prod_{j=k}^{k+r}z_j\prod_{j=k+1}^{k+r+1}y_j.\]
Now, by Theorem \ref{auslander}$(a)$

	\[X_{\tau^{-k+1}P_{r+1}}X_{\tau^{-k}P_{r+1}}=X_{\tau^{-k+1}P_{r+2}}X_{\tau^{-k}P_r}+z^{\underline{\dim}\tau^{-k+1}P_{r+1}}y^{\underline{\dim}\tau^{-k}P_{r+1}}.\]
Recalling that $X_{\tau^{-k+1}P_{r+1}}=u_{r+1}$, $X_{\tau^{-k+1}P_{r+2}}=u_{r+2}$ and $X_{\tau^{-k}P_r}=u_r'$ we get

	\[u_{r+1}X_{\tau^{-k}P_{r+1}}=
				u_{r+2}u_r'+z^{\underline{\dim}\tau^{-k+1}P_{r+1}}y^{\underline{\dim}\tau^{-k}P_{r+1}}
				=u_{r+2}u_r'+\prod_{j=k}^{k+r}z_j\prod_{j=k+1}^{k+r+1}y_j
				= u_{r+1}u_{r+1}'\]
from where we have $X_{\tau^{-k}P_{r+1}}= u_{r+1}'$ for $1\leq r\leq n-k-2$. Since, at vertex $n-k-1$, $\mu_{n-k-2}\cdots\mu_1\mu^k(\widehat{Q})$ has the configuration

	
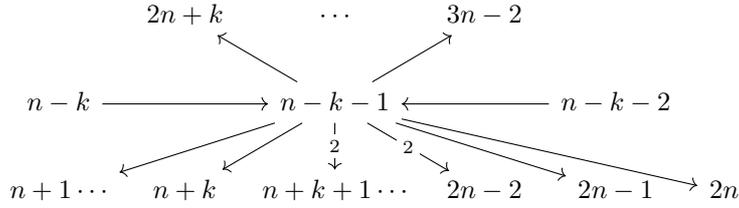
\[\begin{tikzcd}[column sep=tiny]
&& 2n+k & \cdots & 3n-2 && \\
&n-k \ar[rr] && n-k-1 \ar[ul] \ar[ur]\ar[drrr]\ar[drr] \ar[dr, "2" description] \arrow[d, "2" description] \arrow[dl] \ar[dll] && n-k-2 \ar[ll] && \\
& n+1\cdots & n+k & n+k+1 \cdots & 2n-2 & 2n-1 & 2n
\end{tikzcd}
\]
the cluster variable $u_{n-k-1}'$ is obtained by the exchange relation

	\[u_{n-k-1}u_{n-k-1}'=u_{n-k}u_{n-k-2}'+\left(\prod_{j=k}^{n-2}z_j\right)\left(\prod_{j=1}^{k}y_j\right)\left(\prod_{j=k+1}^{n-2}y^2_j\right)y_{n-1}y_n.\]
Now, by Theorem \ref{auslander}$(a)$

	\[X_{\tau^{-k+1}P_{n-k-1}}X_{\tau^{-k}P_{n-k-1}}=X_{\tau^{-k+1}P_{n-k}}X_{\tau^{-k}P_{n-k-2}}+z^{\underline{\dim}\tau^{-k+1}P_{n-k-1}}y^{\underline{\dim}\tau^{-k}P_{n-k-1}}.\]
Recalling that $X_{\tau^{-k+1}P_{n-k-1}}=u_{n-k-1}$, $X_{\tau^{-k+1}P_{n-k}}=u_{n-k}$ and $X_{\tau^{-k}P_{n-k-2}}=u_{n-k-2}'$ we get

	\[u_{n-k-1}X_{\tau^{-k}P_{n-k-1}}=u_{n-k}u_{n-k-2}'+z^{\underline{\dim}\tau^{-k+1}P_{n-k-1}}y^{\underline{\dim}\tau^{-k}P_{n-k-1}}=u_{n-k-1}u_{n-k-1}'\]
from where we have $X_{\tau^{-k}P_{n-k-1}}=u_{n-k-1}'$. Notice that  at vertex $r+1$, $\mu_r\mu_{r-1}\cdots\mu_1\mu^k(\widehat{Q})$ has the following configuration, for $n-k-1\leq r< n-3$,

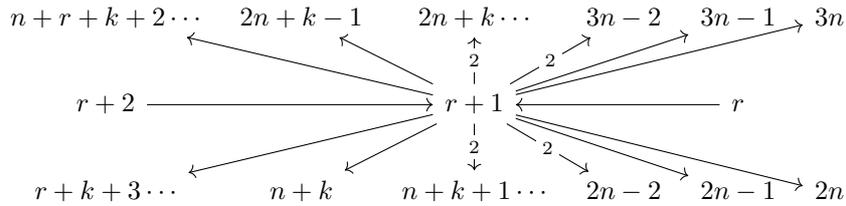
\[\begin{tikzcd}[column sep=tiny]
n+r+k+2 \cdots & 2n+k-1 & 2n+k \cdots & 3n-2 & 3n-1 & 3n \\
 r+2 \ar[rr] && r+1 \ar[drr] \ar[dr, "2" description] \arrow[d, "2" description] \arrow[dl] \ar[dll] \ar[urr] \ar[ur, "2" description] \arrow[u, "2" description] \arrow[ul] \ar[ull] \ar[urrr] \ar[drrr]  && r \ar[ll]& \\
r+k+3 \cdots & n+k & n+k+1 \cdots & 2n-2 & 2n-1 & 2n 
\end{tikzcd}
\] and $\mu_{n-3}\mu_{n-4}\cdots\mu_1\mu^k(\widehat{Q})$ at $n-2$ is of the form:

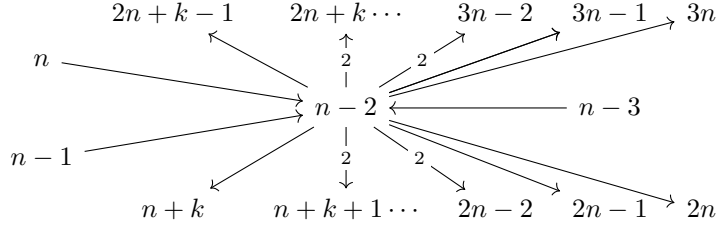
\[\begin{tikzcd}[column sep=tiny, row sep=tiny]
& 2n+k-1 & 2n+k \cdots & 3n-2 & 3n-1 & 3n \\
n \ar[drr] &&&&& \\
&& n-2 \ar[ddl] \ar[dd, "2" description] \ar[ddr, "2" description] \ar[ddrr] \ar[ddrrr] \ar[uurrr] \ar[uurr] \ar[uur, "2" description] \ar[uu, "2" description] \ar[uul] \ar[uurr] && n-3 \ar[ll] \\
n-1 \ar[urr] &&&&& \\
& n+k & n+k+1 \cdots & 2n-2 & 2n-1 & 2n
\end{tikzcd}\]
By using induction in $r$ again it is possible to prove that $\tau^{-k}P_r=u_r'$, $n-k-1\leq r\leq n-2$.

Finally, for the remaining two vertices it is necessary to split the argument into cases $k$ even and $k$ odd; we only prove case $k$ even. Since $\mu_{n-2}\cdots\mu_1\mu^k(\widehat{Q})$ at vertex $n-1$ is

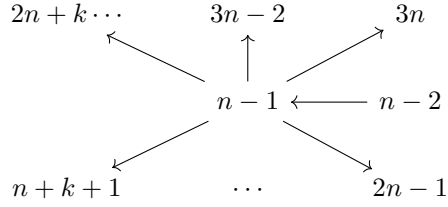
\[\begin{tikzcd}
2n+k \cdots & 3n-2 & 3n \\
& n-1 \ar[ur] \ar[u] \ar[ul] \ar[dr] \ar[dl] & n-2 \ar[l] \\
n+k+1 & \cdots & 2n-1 
\end{tikzcd}
\]
the cluster variable $u'_{n-1}$ is obtained by the exchange relation

\[u_{n-1}u'_{n-1}=u'_{n-2}+\left(\prod_{j=k}^{n-2}z_j\right)z_n\left(\prod_{j=k+1}^{n-1}y_j\right).\]
On the other hand, due to Theorem \ref{auslander}$(a)$
\[X_{\tau^{-k+1}P_{n-1}}X_{\tau^{-k}P_{n-1}}=X_{\tau^{-k}P_{n-2}}+z^{\underline{\dim}\tau^{-k+1}P_{n-1}}y^{\underline{\dim}\tau^{-k}P_{n-1}}.\]
In the view of $X_{\tau^{-k+1}P_{n-1}}=u_{n-1}$ and  $X_{\tau^{-k}P_{n-2}}=u'_{n-2}$, the two above equations yields $X_{\tau^{-k}P_{n-1}}=u'_{n-1}.$ Similarly, it is possible to prove that $X_{\tau^{-k}P_{n}}=u'_{n}$, which completes the proof. 
\end{proof}

We conjecture that results concerning the bijection between mutation classes of the biframed quiver $\widehat{A}_n$ and cluster tilting objects (as stated in Theorem \ref{teo:bij}) and a multiplication formula for cluster characters with coefficients can also be obtained in any Dynkin case; examples can be exploit using Keller's app for mutation \cite{appK}.

\newcommand{\etalchar}[1]{$^{#1}$}
\providecommand{\bysame}{\leavevmode\hbox to3em{\hrulefill}\thinspace}
\providecommand{\MR}{\relax\ifhmode\unskip\space\fi MR }
\providecommand{\MRhref}[2]{%
  \href{http://www.ams.org/mathscinet-getitem?mr=#1}{#2}
}
\providecommand{\href}[2]{#2}

\end{document}